\documentclass[12pt,a4paper]{article}
\usepackage[T2A]{fontenc}
\usepackage[cp1251]{inputenc}
\usepackage[english]{babel}
\usepackage{amsmath,amsfonts,amssymb, amsthm}
\usepackage{cite}
\usepackage{xcolor}

\theoremstyle{plain}
\newtheorem{theorem}{Theorem}[section]
\newtheorem{lemma}{Lemma}[section]
\newtheorem{corollary}{Corollary}[section]
\theoremstyle{definition}

\begin{document}

\title{On factorized groups with permutable subgroups of factors}

\author{Victor S. Monakhov and Alexander A. Trofimuk\\
{\small Department of Mathematics and Programming Technologies,}\\
{\small Francisk Skorina Gomel State University,}\\
{\small Gomel 246019, Belarus}\\
{\small e-mail: victor.monakhov@gmail.com}\\
{\small e-mail: alexander.trofimuk@gmail.com}{\small $^{\dag}$ }
}

\date{}

\maketitle

$$
  \begin{array}{c}
    \textnormal{\small To the 75th anniversary of Mohamed Asaad}
  \end{array}
$$

\bigskip

{\bf Abstract.}
The subgroups $A$ and $B$ of a group~$G$ are called {\rm msp}-permutable, if  the following statements hold: $AB$~is a subgroup of~$G$; the subgroups $P$ and $Q$ are mutually permutable, where $P$~is an arbitrary Sylow $p$-subgroup of~$A$ and $Q$~is an arbitrary Sylow $q$-subgroup of~$B$, ${p\neq q}$. In the present paper, we investigate groups that factorized by two {\rm msp}-permutable subgroups. In particular, the supersolubility of the product of two supersoluble {\rm msp}-permutable subgroups is proved.

\medskip


\medskip


\bigskip

Throughout this paper, all groups are finite and $G$ always denotes a
finite group. We use the standard notations and terminology
of~\cite{BalKniga}. The notation $Y\le X$  means that $Y$ is a subgroup  of a group $X$.

The subgroups $A$ and~$B$ of a group~$G$ are called  {\it mutually (totally) permutable}, if
$UB=BU$ and $AV=VA$ (respectively, $UV=VU$) for all $U\le A$ and $V\le B$.

The idea of totally and mutually permutable subgroups was first initiated by M.~Asaad and A.~Shaalan in~\cite{13}. This direction have since been subject of an in-depth study of many authors. An exhaustive report on this matter appears in~\cite[chapters~4--5]{BalKniga}.

It is quite natural to consider a factorized group $G=AB$ in which certain subgroups of the factors~$A$ and~$B$ are mutually (totally) permutable. In this direction, V.S. Monakhov~\cite{mmz13} obtained the solubility of a group $G=AB$  under the assumption that the subgroups $A$ and $B$ are soluble and the Carter subgroups (Sylow subgroups) of $A$ and of~$B$ are permutable.

We introduce the following

{\bf Definition.} The subgroups $A$ and $B$ of a group~$G$ are called {\rm msp}-permutable, if the following statements hold:

$(1)$ $AB$~is a subgroup of~$G$;

$(2)$ the subgroups $P$ and $Q$ are mutually permutable, where $P$~is an arbitrary Sylow $p$-subgroup of~$A$ and $Q$~is an arbitrary Sylow $q$-subgroup of~$B$, ${p\neq q}$.

In the present paper, we investigate groups that factorized by two {\rm msp}-permutable subgroups. In particular, the supersolubility of the product of two supersoluble {\rm msp}-permutable subgroups is proved.

\section{Preliminaries}\label{pre}

In this section, we give some definitions and basic results which
are essential in the sequel.
A group whose chief factors have prime orders is called {\it supersoluble}.
Recall that a {\it $p$-closed} group is a group with a normal Sylow $p$-subgroup and a {\it $p$-nilpotent} group is a group with a normal Hall $p^{\prime}$-subgroup.

Denote by $G^\prime $, $Z(G)$, $F(G)$ and $\Phi (G)$  the derived subgroup, centre, Fitting and Frattini  subgroups of~$G$, respectively;  $\mathbb P$~the set of all primes.
We use~$E_{p^t}$  to denote an elementary abelian group of order~$p^t$ and $Z_m$ to denote a cyclic group of order~$m$. The semidirect product of a normal subgroup~$A$ and a subgroup~$B$ is written as follows: $A\rtimes B$.

The monographs~\cite{Bal_Clas}, \cite{Doerk} contain the necessary information of the theory of formations. The formations of all nilpotent,  $p$-groups and supersoluble
groups  are denoted by $\mathfrak N$, $\mathfrak N_p$ and $\mathfrak U$, respectively.
A formation $\frak F$ is said to be {\it saturated} if
$G/\Phi (G)\in \frak F$ implies  $G\in \frak F$.
A {\it formation function} is a function $f$ defined on $\mathbb{P}$ such that $f(p)$ is a, possibly empty, formation. A formation $\frak F$ is said to be {\it local} if there exists a formation function $f$ such that $\frak F=\{G \mid G/F_p(G)\in  f(p)\}$. Here $F_p(G)$~is the greatest normal $p$-nilpotent subgroup of~$G$. We write  $\frak F = LF(f)$ and $f$ is a local definition of $\frak F$.
By  \cite[Theorem~IV.3.7]{Doerk}, among all possible local definitions of a local formation~$\frak F$ there exists a unique $f$ such that $f$ is integrated (i.e., $f(p)\subseteq \frak F$ for all  $p\in \mathbb{P}$) and full (i.e., $f(p)=\frak N_pf(p)$ for all  $p\in \mathbb{P}$). Such local definition  $f$ is said to be {\it canonical local definition} of $\frak F$. By \cite[Theorem~IV.4.6]{Doerk}, a formation is saturated if and only if it is local.

A subgroup $H$ of a group $G$ is called $\mathbb P$-{\it subnormal} in~$G$, see~\cite{VVTSMJ10}, if either~${H=G}$, or there is a chain subgroups~$$
H=H_0\le H_1\le \ldots \le H_n=G, \ |H_{i}:H_{i-1}|\in \Bbb P, \
\forall i.
$$

A group $G$ is called {\it $\mathrm{w}$-supersoluble} (widely supersoluble), if every Sylow subgroup of $G$ is $\mathbb P$-subnormal in $G$.
Denote by  $\mathrm{w}\frak U$ the class of all $\mathrm{w}$-supersoluble groups, see~\cite{VVTSMJ10}. In~\cite[Theorem~2.7, Proposition~2.8]{VVTSMJ10}
proved that $\mbox {w}\mathfrak U$ is a subgroup-closed saturated formation
and every group from $\mbox {w}\mathfrak U$ has an ordered Sylow tower of
supersoluble type. By~ \cite[Theorem~B]{Mon_Kn_Res_2013}, \cite[Theorem~2.6]{Mon2016SMJ}, \cite[Theorem~2.13]{VVTSMJ10},
 $G\in \mbox {w}\mathfrak U$ if and only if~$G$ has an ordered Sylow
tower of supersoluble type and every metanilpotent (biprimary)
subgroup of~$G$ is supersoluble.

Denote by  $\mathrm{v}\frak U$ the class of groups all of whose primary cyclic subgroups
are $\mathbb P$-subnormal.
In~\cite[Theorem~B]{Mon_Kn_Res_2013} proved that $\mbox {v}\mathfrak U$ is a subgroup-closed saturated formation and $G\in \mbox {v}\mathfrak U$ if and only if~$G$
has an ordered Sylow tower of supersoluble type  and every biprimary subgroup of $G$ with a cyclic Sylow subgroup is supersoluble.
It is easy to verify that $\frak U\subseteq \mbox {w}\mathfrak U\subseteq \mbox {v}\mathfrak U\subseteq \mathcal{D}$. Here $\mathcal{D}$~is the formation of all groups which have an ordered Sylow tower of supersoluble type.

If $H$ is a subgroup of $G$, then $H_G=\bigcap _{x\in G} H^x$ is called {\it the core} of $H$ in $G$.
If a group $G$ contains a maximal subgroup $M$ with trivial core, then $G$ is said to be {\it primitive} and $M$ is its {\it primitivator}.  A simple check proves the following lemma.

\begin{lemma} \label{l3}
Let $\mathfrak{F}$ be a saturated formation and $G$ be a group. Assume that $G \notin\mathfrak{F}$, but $G/N \in \mathfrak{F}$ for all non-trivial normal subgroups $N$ of $G$. Then $G$ is a primitive group.
\end{lemma}

\begin{lemma} \label{l4} \emph{(\cite[Theorem~15.6]{Doerk})}
Let  $G$ be a soluble primitive group and  $M$ is a primitivator of $G$. Then the following statements hold:

$(1)$ $\Phi (G)=1$;

$(2)$ $F(G)=C_G(F(G))=O_p(G)$ and  $F(G)$  is an elementary abelian subgroup of order~$p^n$ for some prime $p$ and some positive integer~$n$;

$(3)$ $G$ contains a unique minimal normal subgroup $N$ and moreover, $N=F(G)$;

$(4)$ $G=F(G)\rtimes M$ and $O_p(M)=1$.
\end{lemma}

\begin{lemma} \emph{(\cite[Lemma~2.16]{Skiba_2007})} \label{l_skiba}
Let $\frak F$ be a saturated formation containing~$\frak U$ and $G$ be a group with a normal subgroup~$E$ such that~ $G/E\in \frak F$. If $E$ is cyclic, then $G\in  \frak F$.
\end{lemma}

\begin{lemma} \label{ll_4_1_21}
Let $\frak F$~be a formation, $G$~group, $A$ and $B$~subgroups of~$G$ such that $A$ and $B$ belong to $\frak F$. If $[A,B]=1$, then $AB\in \frak F$.
\end{lemma}

\begin{proof}
Since
$$
[A,B]=\langle [a,b] \mid a\in A, \ b\in B\rangle =1,
$$
it follows that $ab=ba$ for all $a\in A$, $b \in B$.
Let
$$
A\times B=\{(a,b)\mid a\in A,\ b\in B\},
$$
$$
(a_1,b_1)(a_2,b_2)=(a_1a_2,b_1b_2), \ \
\forall a_1, a_2\in A,\ b_1, b_2 \in B~
$$
be the external direct product of groups $A$ and $B$.
Since $A\in \frak F$, $B\in \frak F$ and $\frak F$~is a formation,
we have $A\times B\in \frak F.$
Let $\varphi: A\times B\rightarrow AB$~be a function with $\varphi((a, b))=ab$. It is clear that  $\varphi$~is a surjection.
Because
$$
\varphi((a_1,b_1)(a_2,b_2))=\varphi((a_1a_2,b_1b_2))=a_1a_2b_1b_2=
$$
$$
=a_1b_1a_2b_2=\varphi((a_1,b_1))\varphi((a_2,b_2),
$$
it follows that  $\varphi$~is an epimorphism. The core $\mathrm{Ker}~\varphi$ contains  all elements $(a,b)$ such that~$ab=1$. In this case
$
a=b^{-1}\in A\cap B\le Z(G).
$
By the Fundamental Homomorphism Theorem,
$$
A\times B/\mathrm{Ker}~\varphi\cong AB.
$$
Since  $A\times B\in \frak F$ and $\frak F$~is a formation,
$A\times B/\mathrm{Ker}~\varphi \in \frak F$.
Hence $AB\in \frak F$.
\end{proof}

\begin{lemma} \label{l13} \emph{(\cite{Chun})}
Let a group $G = HK$~be the product of subgroups $H$ and $K$. If  $L$ is normal in~$H$ and $L\leq K$, then
$L\leq K_G$.
\end{lemma}

\begin{lemma} \label{l11'}
Let $G=P\rtimes M$~be a primitive soluble group, where $M$~is a primitivator of~$G$ and $P$~is a Sylow $p$-subgroup of~$G$.
Let~$A$ and $B$~be subgroups of~$M$ and $M=AB$.
If $B\leq N_G(X)$ for every subgroup~$X$ of~$P$, then the
following statements hold:

$(1)$ $B$~is a cyclic group of order dividing~$p-1$;

$(2)$ $[A,B]=1$.
\end{lemma}

\begin{proof}
We fix an element $b\in B$. If $x\in P$,
then~$x^b\in \langle x\rangle $, since $B\leq N_G(\langle x\rangle)$
by hypothesis. Hence $x^b=x^{m_x}$, where $m_x$~is a positive integer and
$1\leq m_x\leq p$. If $y\in P\setminus \{x\}$, then
$$
(xy)^b=(xy)^{m_{xy}}=x^{m_{xy}}y^{m_{xy}},\
(xy)^b=x^by^b=x^{m_{x}}y^{m_{y}},\
$$
$$
x^{m_{xy}}y^{m_{xy}}=x^{m_{x}}y^{m_{y}}, \
x^{m_{xy}-m_x}=y^{m_{y}-m_{xy}}=1,  \ m_{xy}=m_x=m_y.
$$
Therefore we can assume that $x^b=x^{n_b}$ for all~$x\in P$, where
${1\leq n_b\leq p}$ and $n_b$~is a positive integer.

Assume that there exist~$d\in B$ and $y\in P\setminus \{1\}$
such that~$y^d=y$. Then $n_d=1$ and $x^d=x$ for all~$x\in P$,
i.e. $d\in C_G(P)=P$ and $d=1$.
Consequently~$B$ is a group automorphism of a group of order~$p$.
Hence $B$~is cyclic of order dividing~$p-1$.

Show that $[A,B]=1$. We fix an element
$[b^{-1},a^{-1}]\in [A,B]$.
Since $P$ is normal in~$G$, it follows that $x^a\in P$ for any~$a\in A$ and any~$x\in P$. Hence
$$
x^{[b^{-1},a^{-1}]}=x^{bab^{-1}a^{-1}}=(x^b)^{ab^{-1}a^{-1}}=
$$
$$
=((x^{n_b})^a)^{b^{-1}a^{-1}}=((x^a)^{n_b})^{b^{-1}a^{-1}}=
((x^a)^{b})^{b^{-1}a^{-1}}=(x)^{abb^{-1}a^{-1}}=x.
$$
Therefore $[b^{-1},a^{-1}]\in C_G(P)=P$. Since
$[A,B]\leq M$, we have $[b^{-1},a^{-1}]\in M\cap P=1$
and~$[A,B]=1$.
\end{proof}

\section{Properties of {\rm msp}-permutable subgroups}

We will say that a group $G$ satisfies the property:

$E_{\pi}$ if $G$ has at least one Hall $\pi$-subgroup;

$C_{\pi}$ if $G$ satisfies $E_{\pi}$ and any two Hall $\pi$-subgroups of $G$ are conjugate in~$G$;

$D_{\pi}$ if $G$ satisfies $C_{\pi}$ and every $\pi$-subgroup of $G$ is contained in some Hall
$\pi$-subgroup of $G$.

Such a group is also called an $E_{\pi}$-group, $C_{\pi}$-group, and $D_{\pi}$-group, respectively.

\begin{lemma} \label{l10}
Let $A$ and $B$~be {\rm msp}-permutable subgroups of~$G$ and~$G=AB$.

$(1)$ If $N$~is a normal subgroup of~$G$, then
$G/N=(AN/N)(BN/N)$~is the {\rm msp}-permutable product of subgroups~$AN/N$ and~$BN/N$.

$(2)$ If $A\le H\le G$, then $H$~is the {\rm msp}-permutable product of subgroups~$A$ and~$H\cap B$.

$(3)$ If $G\in D_{\pi}$, then there exist Hall $\pi$-subgroups~$G_\pi$,
$A_\pi$ and~$B_\pi$ of~$G$, of~$A$ and, of~$B$, respectively, such that $G_{\pi}=A_{\pi}B_{\pi}$~is the
{\rm msp}-permutable product of subgroups~$A_{\pi}$ and~$B_{\pi}$.
\end{lemma}

\begin{proof}
1. Let $p\in \pi(AN/N)$, $X/N$~be a Sylow $p$-subgroup of~$AN/N$
and $P$~be a Sylow $p$-subgroup of~$A$. Then $PN/N=X/N$. Similarly,
if $q\in \pi(BN/N)$ such that $q\neq p$, $Y/N$~is a Sylow $q$-subgroup of~$BN/N$ and $Q$~is a Sylow $q$-subgroup of~$B$. Then $QN/N=Y/N$.
By hypothesis, $P$ and $Q$ are mutually permutable. Hence $X/N$ and $Y/N$ are mutually permutable.

2. By Dedekind's identity, $H=A(H\cap B)$.
Let $A_q$~be a Sylow $q$-subgroup of~$A$,
$R$~be a Sylow  $r$-subgroup of~$H\cap B$, where $q\neq r$, and~$B_r$~be a Sylow $r$-subgroup of~$B$ containing~$R$.
Since $(H\cap B_r)$~is a Sylow $r$-subgroup of~$H\cap B$ and~$R\le H\cap B_r$, it follows that~$R=H\cap B_r$.

Because $A_q$ and $B_r$ are mutually permutable,
we have $A_qU\leq G$ for every subgroup~$U$ of~$R$.

Let $V$~be an arbitrary subgroup of~$A_q$.
Since  $A_q$ and $B_r$ are mutually permutable,
$$
VB_r\le G, \ \ H\cap VB_r=V(H\cap B_r)=VR\le G.
$$
Hence  $A_q$ and $R$ are mutually permutable.

3. By~\cite[Theorem~1.1.19]{BalKniga}, there are Hall $\pi$-subgroups~$G_\pi$, $A_\pi$ and~$B_\pi$ of~$G$, of~$A$, and of~$B$, respectively, such that $G_{\pi}=A_{\pi}B_{\pi}$.
Since $A$ and $B$~are {\rm msp}-permutable, it follows that obviously, $A_{\pi}$ and $B_{\pi}$~are {\rm msp}-permutable.
\end{proof}

%\newpage

\begin{lemma} \label{l11}
Let $A$ and $B$~be {\rm msp}-permutable subgroups of~$G$ and~$G=AB$.
Let~$p,r\in \pi (G)$, $p$~be the greatest prime in~$\pi (G)$ and $r$~be the smallest prime in~$\pi (G)$.
Then the following statements hold:

$(1)$ if $A$ and $B$ are $p$-closed, then~$G$ is $p$-closed;

$(2)$ if $A$ and $B$ are $r$-nilpotent, then~$G$ is $r$-nilpotent;

$(3)$ if  $A$ and $B$ have an ordered Sylow tower of supersoluble type, then~$G$ has an ordered Sylow tower of supersoluble type.
\end{lemma}

\begin{proof}
1. By~\cite[Theorem~1.1.19]{BalKniga}, there are Sylow
$p$-subgroups $P$, $P_1$ and~$P_2$ of $G$, of~$A$, and of~$B$, respectively, such that $P=P_1P_2$. By hypothesis, $P_1$ is normal in~$A$ and $P_2$ is normal in~$B$.
Let $H_1$ and $H_2$~be Hall $p^\prime$-subgroups of~$A$ and of~$B$, respectively, and~$Q$~ be a Sylow $q$-subgroup of~$H_1$, where $q\in \pi (H_1)$. Choose a chain of subgroups
$$
1=Q_0<Q_1<\ldots <Q_{t-1}<Q_t=Q, \ \ |Q_{i+1}:Q_i|=q.
$$
Since $A$ and $B$~are  {\rm msp}-permutable,
we have $P_2Q_i$~is a subgroup of $G$ for every~$i$. Since $|P_2Q_1:P_2|=q$
and~$p>q$, it follows that~$P_2$ is normal in~$P_2Q_1$. Then by induction, we have that~$P_2$ is normal in~$P_2Q$. Because $q$~is an arbitrary prime in~$\pi (H_1)$, it follows that~$P_2$ is normal in~$P_2H_1$ and~$\langle H_1,H_2\rangle \le N_G(P_2)$.
Similarly,~$\langle H_1,H_2\rangle \le N_G(P_1)$. Hence $P=P_1P_2$ is normal in~$G$.

2. Let $R$, $R_1$ and~$R_2$~are Sylow $r$-subgroups of~$G$, of~$A$, and of~$B$, respectively, such that $R=R_1R_2$. Let $K_1$ and $K_2$~be Hall $r^\prime$-subgroups of~$A$ and of~$B$. Let $q\in \pi (G)\setminus \{r\}$, $Q$,
$Q_1$ and~$Q_2$~be Sylow $q$-subgroups of~$G$, of~$A$, and of~$B$, respectively, such that  $Q=Q_1Q_2$. Choose a chain of subgroups
$$
1=V_0<V_1<\ldots <V_{t-1}<V_t=R_1, \ \ |V_{i+1}:V_i|=r.
$$
Since $A$ and $B$~are {\rm msp}-permutable,
$V_iQ_2$~is a subgroup of $G$ for every~$i$. Since $|V_1Q_2:Q_2|=r$
and~$q>r$, it follows that~$Q_2$ is normal in~$V_1Q_2$. Then by induction, we have that~$R_1\le N_G(Q_2)$. By hypothesis,~$A$ is  $r$-nilpotent,
hence~$R_1\le N_G(Q_1)$ and~$R_1\le N_G(Q)$. Similarly,~$R_2\le N_G(Q)$ and~$G$ has a $r$-nilpotent Hall $\{r,q\}$-subgroup~$RQ$. Since $q$~is an arbitrary prime in~$\pi (G)\setminus \{r\}$, it follows that~$G$ is soluble and~$r$-nilpotent by~\cite[Corollary]{TuytKnJAlgebra}.

3. By~(1), we have that a Sylow $p$-subgroup~$P$ is normal in~$G$ for the greatest~$p\in \pi (G)$. By Lemma~\ref{l10}\,(1), $G/P$ is the product of {\rm msp}-permutable subgroups $AN/N$ and $BN/N$. By induction,  $G/N$ has an ordered Sylow tower of supersoluble type, hence~$G$ has an ordered Sylow tower of supersoluble type.
\end{proof}

%\newpage

\begin{theorem} \label{l12}
Let $A$ and $B$~be {\rm msp}-permutable subgroups of~$G$
and~$G=AB$. If $A$ and $B$ are soluble,  then $G$ is soluble.
\end{theorem}

\begin{proof}
We use induction on the order of $G$ and the method of proof from~\cite[Theorem~2]{mmz13}.
Let $N\ne 1$~be a soluble normal subgroup of~$G$.
By Lemma~\ref{l10}\,(1), $G/P$ is the product of soluble {\rm msp}-permutable subgroups
$AN/N$ and $BN/N$. By induction, $G/N$ is soluble, hence~$G$ is soluble.
In what follows, we assume that $G$ contains no non-trivial soluble normal subgroups.

Since~$A$ is soluble, $U=O_s(A)\ne 1$ for some~$s\in \pi (A)$.
If $B$~is an $s$-subgroup of~$G$, then $G=AG_s$, $U\le G_s$
and~$U^G\le (G_s)_G$ by Lemma~\ref{l13}, a contradiction.
Hence~$B$ is not $s$-subgroup of $G$ and let $Q$~be an arbitrary Sylow $q$-subgroup of~$B$, where $q\in \pi (B)\setminus \{s\}$.
Since $A$ and $B$ are {\rm msp}-permutable,
$$
UQ^x=UQ^{ba}=U^a(Q^b)^a=(UQ^b)^a=(Q^bU)^a=Q^xU
$$
for every $x=ba\in G$, where $b\in B$ and $a\in A$.
By~\cite[Theorem 7.2.5]{LS}, $D=U^Q\cap Q^U$ is subnormal in~$G$. Since $U^Q\leq UQ$ and
$UQ$ is soluble, it follows that $D$~is a soluble subnormal subgroup of~$G$ and $D=1$. Hence
$$
[U,Q]\leq [U^Q, Q^U]\leq D=1.
$$
This is true for any Sylow $q$-subgroup of~$B$, therefore~$[U,Q^B]=1$.

Let $H=N_G(U)$. By Dedekind's identity, $H=A(H\cap B)$. By Lemma~\ref{l10}\,(2),
$H$ is the product of soluble {\rm msp}-permutable subgroups $A$ and $H\cap B$. By induction, $H$ is soluble. Since $[U,Q^B]=1$, we have $Q^B\le N_G(U)=H$. Because $G=AB=HB$, $Q^B$ is normal in~$B$ and $Q^B\leq H$, it follows that $Q^B\leq H_G=1$ by Lemma~\ref{l13}, a contradiction.
\end{proof}

\begin{lemma}  \label{l7}
Let $G = G_1G_2$ be the product of {\rm msp}-permutable subgroups $G_1$ and $G_2$.
If a Sylow $p$-subgroup $P$ of~$G$ is normal in~$G$ and abelian, then $P\cap G_i$ is normal in~$G$ for every~$i\in \{1,2\}$.
\end{lemma}

\begin{proof}
Assume that $i,j\in \{1,2\}$ and $i\neq j$. It is clear that $P\cap G_i$ is a Sylow $p$-subgroup of~$G_i$
and $P\cap G_i=(G_i)_p$ is normal in~$G_i$.
Hence~$G_i$ has a Hall $p^\prime$-subgroup~$(G_i)_{p^{\prime}}$. Since $G_i$ and $G_j$ are {\rm msp}-permutable, it follows that $(G_i)_p(G_j)_{p^{\prime}}$~is a subgroup of~$G$
and~$(G_j)_{p^{\prime}}\leq N_G((G_i)_p)$, because every subgroup of~$G$ is
$p$-closed. By hypothesis,~$P$ is abelian, therefore $(G_i)_p$ is normal in~$P$ and
$$
G_j=(G_j)_p(G_j)_{p^{\prime}}=(P\cap G_j)(G_j)_{p^{\prime}}\le N_G((G_i)_p).
$$
Hence $(G_i)_p$ is normal in~$G=G_iG_j=G_1G_2$ for every~$i\in \{1,2\}$.
\end{proof}

\section{Proof of the main theorem}

\begin{theorem} \label{th2} Let $\frak F$~be a subgroup-closed saturated formation such that
$\frak U\subseteq \frak F\subseteq \mathcal{D}$.
Let $G = G_1G_2$ be the product of {\rm msp}-permutable subgroups~$G_1$ and $G_2$. If $G_1, G_2\in \frak F$, then $G\in \frak F$.
\end{theorem}

\begin{proof}
By Lemma~\ref{l11}\,(3), $G$ has an ordered Sylow tower of supersoluble type.
Let $P$ be a Sylow $p$-subgroup of~$G$, where $p$~is the greatest prime in~$\pi(G)$. Then $P$ is normal in~$G$.

Assume that $G\not \in \frak F$. Let $N$~be a non-trivial normal subgroup of~$G$. Hence
$$
G/N=(G_1N/N)(G_2N/N),
$$
$$ G_1N/N\cong G_1/G_1\cap N\in \frak F,\
G_2N/N\cong G_2/G_2\cap N\in \frak F.
$$
By Lemma~\ref{l10}\,(1), $G_1N/N$ and $G_2N/N$ are {\rm msp}-permutable.
Consequently,  $G/N$ satisfies the hypothesis of the theorem, and by induction,  $G/N\in \frak F$. Since $\frak F$ is saturated, $G$ is primitive by Lemma~\ref{l3}. Hence $\Phi(G)=1$,
$G=N\rtimes M$, where $N=C_G(N)=F(G)=\mathrm{O}_p(G)=P$~is a unique minimal normal subgroup of~$G$ by Lemma~\ref{l4}.
Therefore $M$~is a  Hall $p^{\prime}$-subgroup of~$G$ and
$M=(G_1)_{p^\prime}(G_2)_{p^\prime}$ for some Hall ${p^\prime}$-subgroups~$(G_1)_{p^\prime}$ and~$(G_2)_{p^\prime}$ of~$G_1$ and of~$G_2$, respectively.

Suppose that $p$ divides $|G_1|$ and $|G_2|$. By Lemma~\ref{l7},
$P\leq G_1\cap G_2$. Let $P_1\le P$ and $|P_1|=p$.
Since $P\leq G_1$ and~$Q$ permutes with~$P_1$ for every Sylow subgroup~$Q$ of~$(G_2)_{p^\prime}$, we have $P_1(G_2)_{p^{\prime}}\leq G$ and
$(G_2)_{p^{\prime}}\leq N_G(P_1)$. Similarly, since $P\leq G_2$
and~$R$ permutes with~$P_1$ for every Sylow subgroup~$R$ of~$(G_1)_{p^\prime}$, it follows that $P_1(G_1)_{p^{\prime}}\leq G$ and
$(G_1)_{p^{\prime}}\leq N_G(P_1)$.
Hence~$M=(G_1)_{p^\prime}(G_2)_{p^\prime}\leq N_G(P_1)$ and~$P_1$ is normal in~$G$. By Lemma~\ref{l_skiba}, $G\in \frak F$, a contradiction.

Thus $P\leq G_1$ and $G_2$~is a $p^{\prime}$-subgroup of $G$. By Lemma~\ref{l11'}\,(1), $G_2$~is a cyclic group of order dividing~$p-1$.
Hence $G_2\in g(p)$, where $g$~is a canonical local definition of a saturated formation~$\frak U$. Since $\frak U \subseteq \frak F$, we have by~\cite[Proposition~IV.3.11]{Doerk}, $g(p)\subseteq f(p)$, where   $f$~is a canonical local definition of a saturated formation~$\frak F$.
Hence $G_2\in f(p)$. Since $P\leq G_1$, it follows that
$G_1=P\rtimes (G_1)_{p^\prime}$. Because $G_1\in \frak F$ and $F_p(G_1)=P$,
we have $G_1/F_p(G_1)=G_1/P\cong (G_1)_{p^\prime}\in f(p)$.
By Lemma~\ref{l11'}\,(2), $[(G_1)_{p^\prime}, (G_2)_{p^\prime}]=1$.
Since $(G_1)_{p^\prime}\in f(p)$, $(G_2)_{p^\prime}\in f(p)$ and
$f(p)$~is a formation, it follows that by Lemma~\ref{ll_4_1_21},
$G/P\cong M=(G_1)_{p^\prime}(G_2)_{p^\prime}\in f(p)$.
Because $P\in \frak N_p$, we have $G\in \frak F$, a contradiction.
The theorem is proved.
\end{proof}

\begin{corollary}
Let $G = G_1G_2$ be the product of {\rm msp}-permutable subgroups $G_1$ and $G_2$.

$1.$ If  $G_1, G_2\in \frak U$, then $G\in \frak U$.

$2.$ If  $G_1, G_2\in  \mathrm{w}\frak U$, then $G\in \mathrm{w}\frak U$.

$3.$ If  $G_1, G_2\in  \mathrm{v}\frak U$, then $G\in  \mathrm{v}\frak U$.
\end{corollary}

\end{document}